\documentclass[a4paper,10pt,leqno]{amsart}
\usepackage{amsmath,amsfonts,amssymb,amsthm,epsfig,epstopdf,url,array}
\usepackage{amsthm}
\usepackage{mathrsfs}
\usepackage{epsfig}
\usepackage{graphicx}

\newtheorem{thm}{Theorem}[section]
\newtheorem{prop}[thm]{Proposition}

\newtheorem{lem}[thm]{Lemma}
\newtheorem{rmk}[thm]{Remark}

\numberwithin{equation}{section}

\newcommand{\ord}{\text{ord}}
\newcommand{\N}{{\mathbb{N}}}
\newcommand{\z}{{\mathbb{Z}}}
\newcommand{\q}{{\mathbb{Q}}}

\newcommand{\rank}{\operatorname{rank}}

\newcommand{\eq}{\equiv}
\newcommand{\rom}[1]{\uppercase\expandafter{\romannumeral #1\relax}}

\title[Almost regular $m$-gonal forms]{Representations of almost regular $m$-gonal forms $\rom 1$}

\author{Dayoon Park}
\address{Department of Mathematics, The University of Hong Kong, Hong Kong}
\email{pdy1016@hku.hk}
\thanks{}
\begin{document}

\maketitle

\begin{abstract}
It is known that any $m$-gonal form of $\rank n \ge 5$ is almost regular.
    In this article, we study the sufficiently large integers which are represented by (almost regular) $m$-gonal forms of $\rank n \ge 6$.
\end{abstract}

\section{Introduction}
The {\it $m$-gonal number} is defined as a total number of dots to constitute a regular $m$-gon.
We especially call the total number of dots to constitute a regular $m$-gon with $x$ dots for each side
\begin{equation}\label{m number}
P_m(x)=\frac{m-2}{2}(x^2-x)+x
\end{equation}
as {\it $x$-th $m$-gonal number}.
The $m$-gonal number which was designed in 2nd century BC has been an interesting target in number theory for a long time.
A Fermat's conjecture which states that every positive integer written as a sum of at most $m$ $m$-gonal numbers was resolved by Lagrange, Gauss, and Cauchy for $m=4$, $m=3$, and for all $m \ge 3$, respectively.

By definition of $m$-gonal number, only positive integer $x$ would be admitted in (\ref{m number}).
But we may naturally generalize the $m$-gonal number by admitting every rational integer $x$ in (\ref{m number}).
Very recently, as a general version of Fermat's conjecture, the author completed the optimal $\ell_m$ for which every positive integer is written as a sum of $\ell_m$ generalized $m$-gonal numbers for all $m \ge 3$ in \cite{det}.
We call a weighted sum of $m$-gonal numbers 
\begin{equation} \label{m form}
F_m(\mathbf x)=a_1P_m(x_1)+\cdots+a_nP_m(x_n)
\end{equation}
as {\it $m$-gonal form} for $(a_1,\cdots,a_n) \in \N^n$.
For a positive integer $N \in \N$, if the diophantine equation $$F_m(\mathbf x)=N$$ has an integer solution $\mathbf x \in \z^n$, then
we say that {\it $F_m(\mathbf x)$ represents $N$}.
If an $m$-gonal form $F_m(\mathbf x)$ represents every positive integer, then we say that $F_m(\mathbf x)$ is universal.
When $$F_m(\mathbf x) \eq N \pmod r$$ has an integer solution $\mathbf x \in \z^n$ for all $r \in \z$ (which is equivalent with that $F_m(\mathbf x)=N$ has a $p$-adic integer solution $\mathbf x \in \z_p^n$ for every prime $p$), we say that {\it $F_m(\mathbf x)$ locally represents $N$}.
If an $m$-gonal form $F_m(\mathbf x)$ represents all positive integers which are locally represented, then we call {\it $F_m(\mathbf x)$ is regular}.
If an $m$-gonal form $F_m(\mathbf x)$ represents all positive integers which are locally represented but finitely many, then we call {\it $F_m(\mathbf x)$ is almost regular}.
The author and Kim determined every type of regular $m$-gonal form of $\rank \ge 4$ for all $m\ge 14$ with $m \not\eq 0 \pmod 4$ and $m\ge 28$ with $m \eq 0 \pmod 4$ in \cite{reg}.
Chan and Oh showed that every quadratic polynomial of $\rank \ge 5$ is almost regular in \cite{CO}.
So any $m$-gonal form $F_m(\mathbf x)$ of $\rank \ge 5$ would represent every sufficiently large integer which is locally represented by $F_m(\mathbf x)$.
On the other words, there would be a constant $N_{m;(a_1,\cdots,a_n)}>0$ for which
$F_m(\mathbf x)=a_1P_m(x_1)+\cdots +a_nP_m(x_n)$ with $n \ge 5$ represents $N$ provided that
$$\begin{cases}N \ge N_{m;(a_1,\cdots ,a_n)} \\ N \text{ is locally represented by $F_m(\mathbf x)$}. \end{cases}$$

On the other hand, there are some results on the sufficiently large integers which are represented by 
$a_1P_m(x_1) + \cdots + a_nP_m(x_n)$ for some fixed $(a_1,\cdots ,a_n) \in \N^n$ and for all $m \ge 3$.
In 1830, Legendre claimed that for $m \ge 5$, every positive integer $N \ge 28(m-2)^3$ is written as
$$P_m(x_1)+P_m(x_2)+P_m(x_3)+P_m(x_4)+\delta_m(N)$$
for some $(x_1,x_2,x_3,x_4) \in \N^4$, $\delta_m(N)=0$ if $m\not\eq 0 \pmod 2$, and $\delta_m(N) \in \{0,1\}$ if $m \eq 0 \pmod 2$.
Nathanson gave a simple proof of the above Legendre's claim and Cauchy's polygonal number theorem in \cite{N1} and \cite{N2}.
In \cite{MS}, Meng and Sun showed that when $m-2 \eq 0 \pmod 4$, every sufficiently large integer $N \ge 28 (m-2)^3$ is written as 
$$P_m(x_1)+P_m(x_2)+P_m(x_3)+P_m(x_4)$$
for some $(x_1,\cdots,x_4) \in \N^4$,
when $m\not\eq 0 \pmod 4$, every positive integer $N \ge 1628(m-2)^3$ is written as
$$P_m(x_1)+P_m(x_2)+2P_m(x_3)+2P_m(x_4)$$
for some $(x_1,\cdots,x_4) \in \N^4$, 
every positive integer $N \ge 924(m-2)^3$ is written as
$$P_m(x_1)+P_m(x_2)+P_m(x_3)+3P_m(x_4)$$
for some $(x_1,\cdots,x_4) \in \N^4$,
every positive integer $N \ge 1056(m-2)^3$ is written as
$$P_m(x_1)+P_m(x_2)+2P_m(x_3)+4P_m(x_4)$$
for some $(x_1,\cdots,x_4) \in \N^4$.
And Kim \cite{K} showed that 
when $m\not\eq 0 \pmod{4}$, every sufficiently large integer $N \ge \frac{1}{8}(m-2)^3$ is written as
$$P_m(x_1)+P_m(x_2)+P_m(x_3)+P_m(x_4)$$
for some $(x_1,\cdots,x_4) \in \z^4$,
every sufficiently large integer $N \ge \frac{1}{10}(m-2)^3$ is written as
$$P_m(x_1)+P_m(x_2)+P_m(x_3)+2P_m(x_4)$$
for some $(x_1,\cdots,x_4) \in \z^4$,
when $m\not\eq 0 \pmod{4}$, every sufficiently large integer $N \ge \frac{1}{3}(m-2)^3$ is written as
$$P_m(x_1)+P_m(x_2)+2P_m(x_3)+2P_m(x_4)$$
for some $(x_1,\cdots,x_4) \in \z^4$,
and every sufficiently large integer $N \ge \frac{7}{8}(m-2)^3$ is written as
$$P_m(x_1)+2P_m(x_2)+2P_m(x_3)+2P_m(x_4)$$
for some $(x_1,\cdots,x_4) \in \z^4$.

\vskip 1em

In this article, we more generally consider the sufficiently large integers which are represented by an $m$-gonal form $a_1P_m(x_1)+\cdots+a_nP_m(x_n)$
with fixed $(a_1,\cdots,a_n) \in \N^n$.
The following theorem is the main theorem in this article.

\begin{thm}\label{main thm}
Let $F_m(\mathbf x)=a_1P_m(x_1)+\cdots+a_nP_m(x_n)$ be an $m$-gonal form of rank $n \ge 6$.
Then $F_m(\mathbf x)$ represents any integer $N$ provided that
\begin{equation} \label{main eq}
    \begin{cases}N \ge N(a_1,\cdots,a_n)\cdot (m-2)^3 \\ N \text{ is locally represented by } F_m(\mathbf x)\end{cases}
\end{equation}
where $N(a_1,\cdots,a_n)$ is a constant which is dependent only on $a_1,\cdots,a_n$. 
In this sense, the cubic with respect to $m$ is optimal in (\ref{main eq}).
\end{thm}
In order to prove Theorem \ref{main thm}, we use the arithmetic theory of quadratic form.

\vskip 0.3em

In \cite{KL}, Kane and Liu showed that the asymptotically increasing $\gamma_m$ ($\gamma_m$ is the smallest positive integer for which an $m$-gonal form is universal if the $m$-gonal form represents every positive integer up to $\gamma_m$) is bounded by $C_{\epsilon}\cdot m^{7+\epsilon}$ for any $\epsilon >0$ and an absolute constant $C_{\epsilon}>0$ which is dependent on $\epsilon$ by observing the sufficiently large integers which are represented by the $m$-gonal forms of escalator tree of $m$-gonal forms of fixed depth $6$.
Theorem \ref{main thm} may directly improve the upper bound for $\gamma_m$ as $C\cdot (m-2)^3$ for some absolute constant $C$ since every $m$-gonal form of escalator tree of $m$-gonal forms of depth $6$ is locally universal.
Although, the exact growth of $\gamma_m$ on $m$ was already revealed as linear on $m$ by the author and Kim \cite{KP'}.

\section{preliminaries}

Before we move on, we set our languages and notations which are used throughout this article.
Basically, our languages and notations follow conventional that in the arithmetic theory of quadratic form.
A {\it quadratic form over $\z$ (or $\z_p$)} is a $\z$ (or $\z_p$)-module $L$ equipped with a quadratic form $Q:L\rightarrow \z$ (or $\z_p$) and the  corresponding symmetric bilinear map $B(x,y):=\frac{1}{2}(Q(x+y)-Q(x)-Q(y))$.
For a quadratic form $L=\z v_1+\cdots+\z v_n$ of $\rank n$, we usually identify $L$ with its {\it Gram-matrix $\begin{pmatrix}B(v_i,v_j)\end{pmatrix}_{n\times n}$}.
When $L$ is a binary (i.e., $n=2$) quadratic form over $\z$, we adopt the notation $[Q(v_1),B(v_1,v_2),Q(v_2)]$ instead of $\begin{pmatrix}B(v_i,v_j)\end{pmatrix}_{2 \times 2}$.
We write an $m$-gonal form 
\begin{equation} \label{m form'}
a_1P_m(x_1)+\cdots+a_nP_m(x_n)
\end{equation}
 simply as $\left<a_1,\cdots,a_n\right>_m$.
When $m=4$, we conventionally denote the square form (i.e., diagonal quadratic form) of (\ref{m form'}) by $\left<a_1,\cdots, a_n\right>$ instead of $\left<a_1,\cdots,a_n\right>_4$.

\vskip 0.3em

Note that an integer $A(m-2)+B$ is represented by an $m$-gonal form 
$$\sum_{i=1}^na_iP_m(x_i)=\frac{m-2}{2}\{(a_1x_1^2+\cdots+a_nx_n^2)-(a_1x_1+\cdots+a_nx_n)\}+(a_1x_1+\cdots+a_nx_n)$$
over $\z$ (or $\z_p$)
if and only if the diophantine system
\begin{equation}\label{eq1}
\begin{cases}a_1x_1^2+\cdots+a_nx_n^2=2A+B+k(m-4) \\ a_1x_1+\cdots+a_nx_n=B+k(m-2)  \end{cases}\end{equation}
is solvable for some $k \in \z$ (or $\z_p$) over $\z$ (or $\z_p$).
The system (\ref{eq1}) is solvable over $\z$ (or $\z_p$) if and only if
\begin{equation} \label{eq2}
    \left(B+k(m-2)-\left(\sum \limits_{i=2}^na_ix_i\right)\right)^2+\sum\limits_{i=2}^na_1a_ix_i^2=2Aa_1+Ba_1+k(m-4)a_1
\end{equation}
is solvable over $\z$ (or $\z_p$).
The left hand side of the equation (\ref{eq2}) may be written as
$$Q_{a_1 ; \mathbf a}((x_2,\cdots,x_n)-(B+k(m-2))(r_2,\cdots,r_n)) + (B+k(m-2))^2 \cdot \left(1-\sum \limits _{i=2}^na_ir_i\right) $$
where $Q_{a_1 ; \mathbf a}(x_2,\cdots,x_n):=\sum_{i=2}^n(a_1a_i+a_i^2)x_i^2+\sum_{2\le i<j \le n}2a_ia_jx_ix_j$ is a quadratic form and $r_1,\cdots,r_n \in \q$ are the solution for
$$\begin{cases}
(a_1a_2+a_2^2)r_2+a_2a_3r_3+\cdots+a_2a_nr_n=a_2\\
a_2a_3r_2+(a_1a_3+a_3^2)r_3 +\cdots + a_3a_nr_n=a_3 \\
\quad \quad \quad \quad \quad \quad \quad \quad \quad \vdots \\
a_2a_nr_2+a_3a_nr_3+\cdots +(a_1a_n+a_n^2)r_n=a_n.
\end{cases}$$
Note that since 
\begin{equation}\label{n-1 sub}
Q_{a_1 ; \mathbf a}= \begin{pmatrix}
a_1a_2+a_2^2 & a_2a_3 &\cdots & a_2a_n\\
a_2a_3 & a_1a_3+a_3^2 & \cdots & a_3a_n\\
\vdots & \vdots & \ddots  & \vdots \\
a_2a_n & a_3a_n & \cdots &a_1a_n+a_n^2
\end{pmatrix}
\end{equation}
is a sub-quadratic form of rank $n-1$ of the positive definite diagonal quadratic form 
$\sum_{i=1}^na_1a_ix_i^2$ satisfying $a_1x_1+\cdots+a_nx_n=0$,
we have that the determinant of $Q_{a_1 ; \mathbf a}$ is positive.
We denote the non-zero determinant of $Q_{a_1 ; \mathbf a}$ as $d_{a_1;\mathbf a}$.

By Theorem 4.9 in \cite{CO}, since every quadratic polynomial of rank greater than or equal to $5$ is almost regular, when $n \ge 6$, we may get an absolute constant 
\begin{equation} \label{C}
    C(a_1,\cdots,a_n)
\end{equation}
which is dependent on $a_1,\cdots,a_n, r_1,\cdots,r_n$ (so it is essentially dependent only on $a_1,\cdots,a_n$) for which
\begin{equation}\label{qa}
Q_{a_1 ; \mathbf a}((x_2,\cdots,x_n)-(B+k(m-2))(r_2,\cdots,r_n))=c
\end{equation}
has an integer solution $(x_2,\cdots,x_n ) \in \z^{n-1}$
if the equation (\ref{qa}) is locally solvable and $c\ge C(a_1,\cdots,a_n)$.

And then we may observe that an integer $A(m-2)+B$ with $0\le B \le m-3$ satisfying that
$$\text{the system (\ref{eq1}) with $k \in \z$ is locally solvable}$$ 
is represented by the $m$-gonal form $\left<a_1,\cdots, a_n\right>_m$ if 
$$A \ge \frac{1}{2a_1} \left(C(a_1,\cdots,a_n)+(B+k(m-2))^2 \cdot \left(1-\sum _{i=2}^na_ir_i\right)\right)-B-k(m-2) .$$

\vskip 1em

Throughout this article, we show that for any integer $A(m-2)+B$ with $0\le B \le m-3$ which is locally represented by $\left<a_1,\cdots, a_n\right>_m$ with $n \ge 6$, we may choose an integer $k$ in a bounded scope $[0,K(a_1,\cdots,a_n)-1]$ for which
the system (\ref{eq1}) is locally solvable where $K(a_1,\cdots,a_n)$ is an absolute constant which is dependent only on $a_1,\cdots,a_n$.
Which may induce the first argument in Theorem \ref{main thm}.

\vskip 0.3em

In this article, to avoid the appearance of superfluous duplications, we always assume that an $m$-gonal form $\left<a_1,\cdots,a_n\right>_m$ is primitive, i.e., $\gcd(a_1,\cdots,a_n)=1$.
Any unexplained terminology and notation can be found in \cite{O}.

\section{Local representation}

\begin{prop} \label{loc.rep}
Let $F_m(\mathbf x)=a_1P_m(x_1)+\cdots+a_nP_m(x_n)$ be a primitive $m$-gonal form.
\begin{itemize}
    \item [(1) ] When $p$ is an odd prime with $p|m-2$, $F_m(\mathbf x)$ is universal over $\z_p$.
    \item [(2) ] When $m \not\eq 0 \pmod 4$, $F_m(\mathbf x)$ is universal over $\z_2$.
   \item [(3) ] When $p$ is an odd prime with $(p,m-2)=1$, an integer $N$ is represented by $F_m(\mathbf x)$ over $\z_p$ if and only if the integer $8(m-2)N+(a_1+\cdots+a_n)(m-4)^2$ is represented by the diagonal quadratic form $\left<a_1,\cdots,a_n \right>$ over $\z_p$.
    \item [(4) ] When $m \eq 0 \pmod 4$, an integer $N$ is represented by $F_m(\mathbf x)$ over $\z_2$ if and only if the integer $2(m-2)N+(a_1+\cdots+a_n)\left(\frac{m-4}{2}\right)^2$ is represented by the diagonal quadratic form $\left<a_1,\cdots,a_n \right>$ over $\z_2$.
\end{itemize}
\end{prop}
\begin{proof}
(1) One may easily show that for any $N \in \z_p$, $P_m(x)=N$ is solovable over $\z_p$ by using Hensel's Lemma (13:9 in \cite{O}).
Which yields the claim from the assumption that $F_m(\mathbf x)$ is primitive.\\
(2) For any $N \in \z_2$, the equation $P_m(x)=N$ has a $2$-adic integer solution $\frac{m-4-\sqrt{(m-4)^2-8N(m-2)}}{2(m-2)} \in \z_2$.
Which yields the claim from the assumption that $F_m(\mathbf x)$ is primitive. \\
(3) Note that $$N=\sum_{i=1}^na_iP_m(x_i)$$
if and only if
$$8(m-2)N+(a_1+\cdots+a_n)(m-4)^2=\sum_{i=1}^na_i(2(m-2)x_i-(m-4))^2$$
where $2(m-2) \in \z_p^{\times}$.
It completes the proof.\\
(4) Note that $$N=\sum_{i=1}^na_iP_m(x_i)$$
if and only if
$$2(m-2)N+(a_1+\cdots+a_n)\left(\frac{m-4}{2}\right)^2=\sum_{i=1}^na_i\left((m-2)x_i-\frac{m-4}{2}\right)^2$$
where $m-2 \in 2\z_2^{\times}$ and $\frac{m-4}{2} \in 2\z_2$.
It completes the proof.
\end{proof}

\vskip 0.8em

\begin{prop} \label{odd,5}
For an odd prime $p$, if there are more than four units of $\z_p$ in $\{a_1,\cdots,a_n\}$ with $n \ge 6$ by admitting a recursion, then
$$\begin{cases}a_1x_1^2+\cdots+a_nx_n^2=2A+B+k(m-4) \\ a_1x_1+\cdots+a_nx_n=B+k(m-2)  \end{cases}$$
is solvable over $\z_p$
for any $A,B,k \in \z$.
\end{prop}
\begin{proof}
First assume that $a_1+\cdots+a_n \not\eq 0\pmod p$.
Then by using 82:15 and 91:9 in \cite{O}, we may obtain that for any $A,B,k$,
the binary quadratic $\z_p$-lattice 
$[a_1+\cdots+a_n, B+k(m-2), 2A+B+k(m-4)] \otimes \z_p$
 is represented as 
$$\begin{pmatrix}1 & 1& \cdots & 1 \\ x_1 & x_2 & \cdots & x_n \end{pmatrix} 
\begin{pmatrix} a_1 & 0 & \cdots & 0 \\ 0 & a_2 &\cdots & 0 \\\vdots & \vdots & \ddots & \vdots \\ 0 & 0 & \cdots & a_n \end{pmatrix}
\begin{pmatrix} 1&x_1\\ 1 & x_2 \\ \vdots & \vdots \\ 1 & x_n  \end{pmatrix}
$$
for some $(x_1,\cdots,x_n) \in \z_p^n$ since every quaternary unimodular $\z_p$-lattice is universal.
Which yields the claim.

Now assume that $a_1+\cdots+a_n \eq 0\pmod p$.
Without loss of generality, let $0=\ord_p(a_1)$.
Then through same arguments with the above, one may induce that
$$\begin{cases}a_2x_2^2+\cdots+a_nx_n^2=2A+B+k(m-4) \\ a_2x_2+\cdots+a_nx_n=B+k(m-2)  \end{cases}$$
is solvable over $\z_p$
for any $A,B,k$ since $a_2+\cdots+a_n \not\eq 0 \pmod p$ and every ternary unimodular $\z_p$-lattice is universal.
This completes the proof.
\end{proof}

\vskip 0.8em

\begin{prop} \label{odd,nd12}
Let $p$ be an odd prime with $(p,m-2)=1$.
If an integer $A(m-2)+B$ is represented by $\left<a_1,\cdots,a_n\right>_m$ of $\rank n \ge 6$ over $\z_p$,
then there is a residue $k_{m;\mathbf a, A,B}(p)$ in $\z/p^{2\ord_p(d_{a_1;\mathbf a})}\z$ for which
$$\begin{cases}a_1x_1^2+\cdots+a_nx_n^2=2A+B+k(m-4) \\ a_1x_1+\cdots+a_nx_n=B+k(m-2)  \end{cases}$$
is solvable over $\z_p$ for any $k \eq k_{m;\mathbf a,A,B}(p) \pmod { p^{2\ord_p(d_{a_1;\mathbf a})}}$.
\end{prop}
\begin{proof}
It would be enough to show that if
$$\begin{cases}a_1x_1^2+\cdots+a_nx_n^2=2A+B+k(m-4) \\ a_1x_1+\cdots+a_nx_n=B+k(m-2)  \end{cases}$$
is solvable over $\z_p$, then
$$\begin{cases}a_1x_1^2+\cdots+a_nx_n^2=2A+B+k'(m-4) \\ a_1x_1+\cdots+a_nx_n=B+k'(m-2)  \end{cases}$$
is also solvable over $\z_p$ for any $k' \in k+p^{2\ord_p(d_{a_1;\mathbf a})}\z_p$.
In other ways, we show that if the equation
$$Q_{a_1 ; \mathbf a}(\mathbf x-(B+k(m-2))\mathbf r)=a_1(2A+B+k(m-4))-(B+k(m-2))^2 \cdot \left(1-\sum _{i=2}^na_ir_i\right)$$
is solvable over $\z_p$, then the equation
$$Q_{a_1 ; \mathbf a}(\mathbf x-(B+k'(m-2))\mathbf r)=a_1(2A+B+k'(m-4))-(B+k'(m-2))^2 \cdot \left(1-\sum _{i=2}^na_ir_i\right)$$
is also solvable over $\z_p$ for any $k' \in k+p^{2\ord_p(d_{a_1;\mathbf a})}\z_p$.
In virtue of the {\it Jordan decomposition} (see 91 in \cite{O}), we may take $T \in GL_{n-1}(\z_p)$ for which
$$T^t\cdot  \begin{pmatrix}
a_1a_2+a_2^2 & a_2a_3 &\cdots & a_2a_n\\
a_2a_3 & a_1a_3+a_3^2 & \cdots & a_3a_n\\
\vdots & \vdots & \ddots  & \vdots \\
a_2a_n & a_3a_n & \cdots &a_1a_n+a_n^2
\end{pmatrix} \cdot T=\begin{pmatrix}
a_2' & 0 &\cdots & 0\\
0 & a_3' &\cdots & 0\\
\vdots & \vdots & \ddots  & \vdots \\
0 & 0 & \cdots &a_n'
\end{pmatrix}.$$
Which allows us to consider the simple diagonal form
$$\sum_{i=2}^na_i'(x_i-(B+k(m-2))r_i')^2$$
instead of $Q_{a_1 ; \mathbf a}(\mathbf x-(B+k(m-2))\mathbf r)$ where $\mathbf r'=T\cdot \mathbf r$.
Note that
$$\left(\z_p-(B+k'(m-2))r_i'\right)^2 =((B+k(m-2))r_i')\cdot \z_p+((B+k(m-2))r_i')^2 $$
for any $k' \in k+p^{2\ord_p(d_{a_1;\mathbf a})}\z_p$
if $\ord_p((B+k(m-2))r_i')<0$.
So if there is an $i$ such that $\ord_p((B+k(m-2))r_i')<0$, then 
the statement would hold
because 
$$\begin{cases}\ord_p(r_i')\ge -\ord_p(d_{a_1;\mathbf a}) \\ \ord_p(a_i') \le \ord_p(d_{a_1;\mathbf a}) \end{cases}$$
for all $i$.

Now assume that $\ord_p((B+k(m-2))r_i')\ge 0$ (i.e., $(B+k(m-2))r_i' \in \z_p$) for all $i$.
If $$a_1(2A+B+k(m-4))+(B+k(m-2))^2 \cdot \left(1-\sum _{i=2}^na_ir_i\right) <\ord_p(d_{a_1;\mathbf a}),$$ then 
$a_1(2A+B+k'(m-4))+(B+k'(m-2))^2 \cdot (1-\sum _{i=2}^na_ir_i)$ is equivalent with 
$a_1(2A+B+k(m-4))+(B+k(m-2))^2 \cdot (1-\sum _{i=2}^na_ir_i)$ 
up to $\z_p$-unit square for any $k' \in k+p^{2\ord_p(d_{a_1;\mathbf a})}\z_p$.
So we obtain the claim.
If $$a_1(2A+B+k(m-4))+(B+k(m-2))^2 \cdot \left(1-\sum _{i=2}^na_ir_i\right) \ge \ord_p(d_{a_1; \mathbf a}),$$ then 
$a_1(2A+B+k'(m-4))+(B+k'(m-2))^2 \cdot (1-\sum _{i=2}^na_ir_i) \in p^{\ord_p(d_{a_1; \mathbf a})}\z_p$ for any $k' \in k+p^{2\ord_p(d_{a_1;\mathbf a})}\z_p$.
On the other hand, the diagonal quadratic $\z_p$-lattice $$\left<a_2',\cdots, a_n'\right>\otimes \z_p$$ of $\rank n-1 \ge 5$ is $p^{\ord_p(d_{a_1;\mathbf a})}\z_p$-universal since $\max\{\ord_p(a_2'),\cdots,\ord_p(a_n') \} \le \ord_p(d_{a_1;\mathbf a})$.
This completes the proof.
\end{proof}

\vskip 0.8em

\begin{prop} \label{odd,d}
Let $p$ be an odd prime with $p|m-2$.
Then for any integer $A(m-2)+B$, there  is a residue $k_{m;\mathbf a, A,B}(p)$ in $\z/p^{\max \limits_{1\le i \le n} \{\ord_p(a_i)\}} \z$ for which 
$$\begin{cases}a_1x_1^2+\cdots+a_nx_n^2=2A+B+k(m-4) \\ a_1x_1+\cdots+a_nx_n=B+k(m-2)  \end{cases}$$
with $n \ge 6$ is solvable over $\z_p$ for any $k \eq k_{m;\mathbf a,A,B}(p) \pmod { p^{\max \limits_{1\le i \le n} \{\ord_p(a_i)\}}}$.
\end{prop}
\begin{proof}
First assume that $a_1+\cdots+a_n \not\eq 0 \pmod p$.
According to Hensel's Lemma (13:9 in \cite{O}), we may take a residue $k_{m;\mathbf a, A,B}(p)$ in $\z/p^{\max \limits_{1\le i \le n}\{\ord_p(a_i)\}} \z$ satisfying 
$$d_{m;\mathbf a, A, B,k} \eq 0 \pmod{p^{\max \limits_{1\le i \le n} \{\ord_p(a_i)\}}}$$
for any $k \eq k_{m;\mathbf a, A,B}(p) \pmod{p^{\max \limits_{1\le i \le n} \{\ord_p(a_i)\}}}$.
By the uniqueness of Jordan decomposition up to {\it Jordan type} (see 91:9 in \cite{O}), the orthogonal complement of the unary unimodular $\left<a_1+\cdots + a_n\right>\otimes \z_p$ in $\left<a_1,\cdots,a_n\right>\otimes \z_p$ of $\rank n \ge 6$ would be $p^{\max \limits_{1\le i \le n} \{\ord_p(a_i)\}}\z_p$-universal.
By using 82:15 and 91:9 in \cite{O}, we may see that for any integer
$k \eq k_{m;\mathbf a, A,B}(p) \pmod{p^{\max \limits_{1\le i \le n} \{\ord_p(a_i)\}}}$, the binary quadratic $\z_p$-lattice
$[a_1+\cdots+a_n, B+k(m-2), 2A+B+k(m-4)] \otimes \z_p$
(which is equivalent with $$\left<a_1+\cdots + a_n , (a_1+\cdots + a_n)^{-1} \cdot d_{m;\mathbf a, A, B,k} \right> \otimes \z_p$$ with the discriminant $d_{m;\mathbf a, A, B,k} \eq 0 \pmod{p^{\max \limits_{1\le i \le n} \{\ord_p(a_i)\}}}$) has a representation by $\left<a_1,\cdots,a_n\right>\otimes \z_p$ of the type
$$\begin{pmatrix}1 & 1& \cdots & 1 \\ x_1 & x_2 & \cdots & x_n \end{pmatrix} 
\begin{pmatrix} a_1 & 0 & \cdots & 0 \\ 0 & a_2 &\cdots & 0 \\\vdots & \vdots & \ddots & \vdots \\ 0 & 0 & \cdots & a_n \end{pmatrix}
\begin{pmatrix} 1&x_1\\ 1 & x_2 \\ \vdots & \vdots \\ 1 & x_n  \end{pmatrix}
$$
for some $(x_1,\cdots,x_n) \in \z_p^n$. 
This completes the claim.

When $a_1+\cdots+a_n \eq 0 \pmod p$, without loss of generality, let $0=\ord_p(a_1)$.
Then one may show that
$$\begin{cases}a_2x_2^2+\cdots+a_nx_n^2=2A+B+k(m-4) \\ a_2x_2+\cdots+a_nx_n=B+k(m-2)  \end{cases}$$
is solvable over $\z_p$ for any $k \eq k_{m;\mathbf a,A,B}(p) \pmod { p^{\max \limits_{1\le i \le n} \{\ord_p(a_i)\}}}$ where $k_{m;\mathbf a, A,B}(p)$ is the residue in  $\z/p^{\max \limits_{1\le i \le n}\{\ord_p(a_i)\}} \z$ satisfying 
$$d_{m;\mathbf a, A, B,k_{m,\mathbf a, A,B}(p)} \eq 0 \pmod{p^{\max \limits_{1\le i \le n} \{\ord_p(a_i)\}}}$$ through same arguments with the above.
This completes the proof.
\end{proof}

\vskip 0.8em

\begin{prop} \label{2,nd}
Let $m\eq 0 \pmod 4$.
If an integer $A(m-2)+B$ is represented by an $m$-gonal form $\left<a_1,\cdots,a_n\right>_m$ of $\rank n \ge 6$ over $\z_2$,
then there is a residue $k_{m;\mathbf a, A,B}(2)$ in $\z/8\cdot2^{2\ord_2(d_{a_1;\mathbf a})}\z$ for which for any $k \eq k_{m;\mathbf a,A,B}(2) \pmod { 8\cdot 2^{\ord_2(d_{a_1,\mathbf a})}}$,
$$\begin{cases}a_1x_1^2+\cdots+a_nx_n^2=2A+B+k(m-4) \\ a_1x_1+\cdots+a_nx_n=B+k(m-2)  \end{cases}$$
is solvable over $\z_2$.
\end{prop}
\begin{proof}
After Jordan decomposing (see 91 in \cite{O}) $$\begin{pmatrix} 
a_1a_2+a_2^2 & a_2a_3 &\cdots & a_2a_n\\
a_2a_3 & a_1a_3+a_3^2 & \cdots & a_3a_n\\
\vdots & \vdots & \ddots  & \vdots \\
a_2a_n & a_3a_n & \cdots &a_1a_n+a_n^2
\end{pmatrix}\otimes \z_2,$$
through similar arguments with the proof of Proposition \ref{odd,nd12}, one may prove this by showing that if the equation
$$Q_{a_1 ; \mathbf a}(\mathbf x-(B+k(m-2))\mathbf r)=a_1(2A+B+k(m-2))-(B+k(m-2))^2 \cdot \left(1-\sum _{i=2}^na_ir_i\right)$$
is solvable over $\z_2$, then the equation
$$Q_{a_1 ; \mathbf a}(\mathbf x-(B+k'(m-2))\mathbf r)=a_1(2A+B+k'(m-2))-(B+k'(m-2))^2 \cdot \left(1-\sum _{i=2}^na_ir_i\right)$$
is also solvable over $\z_2$ for any $k' \in k+8\cdot2^{2\ord_2(d_{a_1;\mathbf a})}\z_2$.
One may use that 
$Q_{a_1;\mathbf a} \otimes \z_2$ is $2\cdot 2^{\ord_2(d_{a_1;\mathbf a})}\z_2$-universal,
for $u,u' \in \z_2^{\times}$,
$$u' \in u (\z_2)^2 \text { if and only if } u\eq u' \pmod 8,$$
for $q \in \q_2$,
$$\begin{cases}
\left(\z_2-q \right)^2= 4q\z_2 + q^2  & \text{if } \ord_2(q)=-1\\
\left(\z_2-q\right)^2= 2q\z_2 + q^2 & \text{if } \ord_2(q)<-1,
\end{cases}$$
and
$$\begin{cases}
\begin{pmatrix}
0 & 1 \\ 1 & 0
\end{pmatrix}
=

\begin{pmatrix}
\frac{1}{2} & \frac{1}{2} \\ \frac{1}{2} &- \frac{1}{2}
\end{pmatrix}^t

\begin{pmatrix}
2 & 0 \\ 0 & -2
\end{pmatrix}

\begin{pmatrix}
\frac{1}{2} & \frac{1}{2} \\ \frac{1}{2} &- \frac{1}{2}
\end{pmatrix}

\\

\begin{pmatrix}
2 & 1 \\ 1 & 2
\end{pmatrix}
=\begin{pmatrix}
1 & \frac{1}{2} \\ 0 & \frac{1}{2}
\end{pmatrix}^t

\begin{pmatrix}
2 & 0 \\ 0 & 6
\end{pmatrix}

\begin{pmatrix}
1 & \frac{1}{2} \\ 0 & \frac{1}{2}
\end{pmatrix}.

\end{cases}$$

\end{proof}

\vskip 0.8em

\begin{prop} \label{2,d}
Let $m \not\eq 0 \pmod 4$.
For any integer $A(m-2)+B$, there is a residue $k_{m;\mathbf a,A,B}(2)$ in $\z /2\cdot2^{\max \limits_{1\le i \le n} \{\ord_2(a_i)\}} \z$ for which for any $k \eq k_{m;\mathbf a,A,B}(2) \pmod {2\cdot 2^{\max \limits_{1\le i \le n} \{\ord_2(a_i)\}}}$,
$$\begin{cases}
    a_1x_1^2+\cdots+a_nx_n^2=2A+B+k(m-4)\\
    a_1x_1+\cdots+a_nx_n=B+k(m-2)
    \end{cases}$$ with $n \ge 6$
    is solvable over $\z_2$.
\end{prop}
\begin{proof}
Without loss of generality, let $0=\ord_2(a_1) \le \ord_2(a_2) \le \cdots \le \ord_2(a_n)$.
First, assume that $a_1+\cdots+a_n$ is odd.
For $m\not\eq 0 \pmod4$ and $a_1+\cdots+a_n$ is odd, we may take $k_2 \in \z_2$ for which
$$d_{m;\mathbf a,A,B,k_2}:=(a_1+\cdots+a_n)(2A+B+k_2(m-4))-(B+k_2(m-2))^2=0.$$
Put the residue $k_{m;\mathbf a,A,B}(2)$ in $\z /2\cdot2^{\max \limits_{1\le i \le n} \{\ord_2(a_i)\}}\z$ as $k_{m;\mathbf a,A,B}(2) \eq k_2  \pmod {2\cdot2^{\max \limits_{1\le i \le n} \{\ord_2(a_i)\}}\z_2}$.
For $k \eq k_{m;\mathbf a,A,B}(2) \pmod {2\cdot 2^{\max \limits_{1\le i \le n} \{\ord_2(a_i)\}}}$, unary unimodular $\left<a_1+\cdots+a_n\right>\otimes \z_2$ splits the binary quadratic $\z_2$-lattice
$$\begin{pmatrix}a_1+\cdots+a_n & B+k(m-2) \\ B+k(m-2) & 2A+B+k(m-4) \end{pmatrix} \otimes \z_2$$
with discriminant $d_{m;A,B,k} \eq 0 \pmod{2\cdot2^{\max \limits_{1\le i \le n} \{\ord_2(a_i)\}}}$.
On the other words, $[a_1+\cdots+a_n, B+k(m-2),2A+B+k(m-4)]\otimes \z_2$ is isometric to diagonal $\left<a_1+\cdots+a_n,(a_1+\cdots+a_n)^{-1}\cdot d_{m;A,B,k}\right>\otimes \z_2$.
The orthogonal complement of unary unimodular $\left<a_1+\cdots+a_n\right>\otimes \z_2$ in $\left<a_1,\cdots,a_n\right>\otimes \z_2$ of $\rank n \ge 6$ would represent every $2$-adic integer in $2\cdot 2^{\max \limits_{1\le i \le n} \{\ord_2(a_i)\}}\z_2$.
From 82:15 in \cite{O}, we may see that $$\begin{pmatrix}a_1+\cdots+a_n & B+k(m-2) \\ B+k(m-2) & 2A+B+k(m-4) \end{pmatrix} \otimes \z_2$$
has a representation by $\left<a_1,\cdots,a_n\right>\otimes \z_2$ of the type
$$\begin{pmatrix}1 & 1& \cdots & 1 \\ x_1 & x_2 & \cdots & x_n \end{pmatrix} 
\begin{pmatrix} a_1 & 0 & \cdots & 0 \\ 0 & a_2 &\cdots & 0 \\\vdots & \vdots & \ddots & \vdots \\ 0 & 0 & \cdots & a_n \end{pmatrix}
\begin{pmatrix} 1&x_1\\ 1 & x_2 \\ \vdots & \vdots \\ 1 & x_n  \end{pmatrix}
$$
for some $(x_1,\cdots,x_n) \in \z_2^n$.
It completes the proof when $a_1+\cdots+a_n$ is odd.

When $a_1+\cdots+a_n$ is even, one may show that
$$\begin{cases}
    a_2x_2^2+\cdots+a_nx_n^2=2A+B+k(m-4)\\
    a_2x_2+\cdots+a_nx_n=B+k(m-2)
    \end{cases}$$
    is solvable over $\z_2$ for any integer $k \eq k_{m;\mathbf a,A,B}(2) \pmod {2\cdot2^{\max \limits_{1\le i \le n} \{\ord_2(a_i)\}}}$ where $k_{m;\mathbf a,A,B}(2)$ is the residue in $\z/2\cdot 2^{\max \limits_{1\le i \le n} \{\ord_2(a_i)\}}\z$ satisfying 
    $$d_{m;(a_2,\cdots,a_n),A,B,k_{m;\mathbf a,A,B}(2)} \eq 0 \pmod{2\cdot2^{\max \limits_{1\le i \le n} \{\ord_2(a_i)\}}}$$
    similarly with the above.
\end{proof}

\vskip 0.8em

\section{Representing integer which is locally represented}
\begin{lem}\label{main lem}
There is a constant $K(a_1,\cdots,a_n)>0$ which is dependent only on $a_1,\cdots, a_n$ for which
for any integer $A(m-2)+B$ with $0\le B \le m-3$ which is locally represented by an $m$-gonal form $\left<a_1,\cdots, a_n\right>_m$ of $\rank n \ge 6$, there is $0 \le k \le K(a_1,\cdots,a_n)-1$ such that
$$\begin{cases}a_1x_1^2+\cdots+a_nx_n^2=2A+B+k(m-4) \\ a_1x_1+\cdots+a_nx_n=B+k(m-2)  \end{cases}$$
is locally solvable.
\end{lem}
\begin{proof}
Define the finite set $T$ as the set of all odd primes $p$ for which the number of units of $\z_p$ in $\{a_1,\cdots, a_n\}$ by admitting a recursion is less than $5$.
Then the lemma holds for $$K(a_1,\cdots,a_n):=\prod \limits _{p \in \{2\} \cup T}p^{\max\{\ord_p(8\cdot (d_{a_1;\mathbf a})^2),\ord_p(2a_1),\cdots,\ord_{p}(2a_n)\}}$$
by Proposition \ref{loc.rep}, Proposition \ref{odd,5}, Proposition \ref{odd,nd12}, Proposition \ref{odd,d}, Proposition \ref{2,nd}, and Proposition \ref{2,d}.
\end{proof}

\vskip 1em

\begin{proof}[proof of Theorem \ref{main thm}]
Recall that for an integer $A(m-2)+B$ with $0\le B \le m-3$ such that
the system (\ref{eq1}) is locally solvable, if 
$$A \ge \frac{1}{2a_1} \left(C(a_1,\cdots,a_n)+(B+k(m-2))^2 \cdot \left(1-\sum _{i=2}^na_ir_i\right)\right)-B-k(m-2)$$
where $C(a_1,\cdots,a_n)$ is the constant in (\ref{C}), then $A(m-2)+B$ is represented by the $m$-gonal form $\left<a_1,\cdots,a_n\right>_m$.
Now we define $N(a_1,\cdots,a_n)$ as 
$$\frac{1}{2a_1}\left(C(a_1,\cdots,a_n)+K(a_1,\cdots,a_n)^2\cdot \left(1-\sum_{i=2}^na_ir_i\right)\right).$$
Then by Lemma \ref{main lem}, we complete the first claim.

One may notice that
$$(a_1+\cdots+a_n)(a_1x_1^2+\cdots+a_nx_n^2) \ge (a_1x_1+\cdots+a_nx_n)^2$$
and 
$$a_1P_m(x_1)+\cdots+a_nP_m(x_n) \eq a_1x_1+\cdots+a_nx_n \pmod{m-2}.$$
Therefore an integer $A(m-2)+B$ with $0\le B \le m-2$ would be never represented by $\left<a_1,\cdots,a_n\right>_m$ if 
$$A<\begin{cases} \frac{B^2}{2(a_1+\cdots+a_n)}+\frac{B}{2} & \text{when } B \le \frac{m-2}{2}\\
\frac{(B-(m-2))^2}{2(a_1+\cdots+a_n)}+\frac{(B-(m-2))}{2} & \text{when } B > \frac{m-2}{2}.
\end{cases}$$
Which implies that the cubic with respect to $m$ in (\ref{main eq}) is optimal.
\end{proof}

\begin{rmk}
Recall that every $m$-gonal form of $\rank \ge 5$ is almost regular.
Indeed, the $5$ is optimal, i.e., there is a non-almost regular $m$-gonal form of $\rank 4$.
By Proposition \ref{loc.rep}, the $m$-gonal $\left<1,1,1,1\right>_m$ is locally universal for all $m \ge 3$, but $\left<1,1,1,1\right>_m$ with $m \eq 0 \pmod 4$ and $m > 8$ is not almost universal (see Theorem 1.2 in \cite{K}).
Theorem \ref{main thm} is missing the sufficiently large integers which are represented by (almost regular) $m$-gonal form of $\rank 5$.
Filling out the remaining answer on sufficiently large integers which are represented by $m$-gonal form of $\rank 5$ could be a challenging problem.
\end{rmk}


\begin{thebibliography}{abcd}


\bibitem {CO} W. K. Chan, B.-K. Oh, {\em Representations of integral quadratic polynomials},
Contemp. Math. 587 (2013), 31–46.

\bibitem {KL} B. Kane, J. Liu, {\em Universal sums of $m$-gonal numbers}, Int. Math. Res. Not. (2019)

\bibitem {KP'} B. M. Kim, D. Park, {\em A finiteness theorem for universal $m$-gonal forms}, preprint.

\bibitem {reg} B. M. Kim, D. Park, {\em Regular $m$-gonal forms}, preprint.

\bibitem {K} D. Kim, {\em Weighted sums of generalized polygonal numbers with coefficients 1 or 2}, preprint.

\bibitem {O1}  O. T. O'Meara, {\em The Integral representations of quadratic forms over local fields}, Amer. J. of Math, \textbf{80}(1958), 843-878.

\bibitem {O} O. T. O'Meara, {\em Introduction to quadratic forms}, Springer-Verlag, New York, (1963).

\bibitem {MS} X.-Z. Meng and Z.-W. Sun, {\em Sums of four polygonal numbers with coefficients}, Acta Arith.
180(2017), 229-249.

\bibitem {N1} M. B. Nathanson, {\em A short proof of Cauchy’s polygonal number theorem}, Proc. Amer.
Math. Soc. 99 (1987), 22–24.

\bibitem {N2} M. B. Nathanson, {\em Additive Number Theory: The Classical Bases}, Grad. Texts in
Math., Vol. 164, Springer, New York, (1996).


\bibitem {P} D. Park, {\em Universal m-gonal forms}, Ph.D. Thesis, Seoul National University. (2020).

\bibitem {det} D. Park, {\em Determining universality of $m$-gonal form with first five coefficients}, preprint.



 

\end{thebibliography}
\end{document}